\documentclass{bcp}
\usepackage{xy}
\usepackage{graphicx}
\usepackage{tikz}
\usetikzlibrary{arrows,calc}
\xyoption{all}
\usepackage{amsmath,amssymb,amsthm,amsfonts}
\newtheorem{defi}{Definition}[section]
\newtheorem{thm}[defi]{Theorem}

\newtheorem{lm}[defi]{Lemma}
\newtheorem{cor}[defi]{Corollary}
\newtheorem{pro}[defi]{Proposition}
\newcommand{\colim}{\operatorname{\rm colim}}
\newcommand{\proj}[2]{\B #1 P_\mathcal{T}^#2}

\newcommand{\id}{\operatorname{\rm id}}

\newcommand{\B}[1]{{\mathbb #1}}

\begin{document}
\abbrevauthors{J. Rudnik}
\abbrevtitle{The $K$-theory of the triple-Toeplitz deformation of $\B CP^2$}
\title{The $K$-theory of\\ the triple-Toeplitz deformation of\\
the complex projective plane} 
\author{Jan Rudnik}
\address{Instytut Matematyczny, Polska Akademia Nauk,\\
ul.~\'Sniadeckich 8, Warszawa, 00-956 Poland\\
Email: jrudnik@impan.pl}
\maketitlebcp

\begin{abstract}
We consider a family 
$\pi^i_j\colon B_i\rightarrow B_{ij}=B_{ji}$, $i,j\in \{1,2,3\}$,
$i\neq j$, of $C^*$-epimorphisms assuming that it
satisfies the cocycle condition. Then we show how
 to compute the $K$-groups of the multi-pullback $C^*$-algebra
of such a family, and examplify it in the case
of the triple-Toeplitz deformation of $\B CP^2$.
\end{abstract}


\section*{Introduction}
Starting from the affine covering of a projective space, a new type of 
noncommutative deformations of complex projective spaces was introduced 
in~\cite{pmh}.  Therein, the complex projective space $\mathbb{C}P^n$ is
presented as a natural gluing of polydiscs, dualized to the 
multi-pullback  $C^*$-algebra, and deformed to a multi-pullback
of tensor powers of Toeplitz algebras. The case of $n=1$ 
was analyzed in detail
 in~\cite{pms}, and called the mirror
quantum sphere. In particular,
its $K$-groups were easily determined. 

The goal of this note is to
determine the $K$-groups in the case $n=2$, which requires some tools.
The $C^*$-algebra of the mirror quantum sphere is simply a pullback
$C^*$-algebra, so that its $K$-theory is immediately computable by the 
Mayer-Vietoris six-term exact sequence. 
The $C^*$-algebra of the triple-Toeplitz deformation of $\mathbb{C}P^2$
is a triple-pullback $C^*$-algebra, and it turns out 
that, in order to apply 
(three times) the 
Mayer-Vietoris six-term exact sequence, we need to check the cocycle
condition. 

We begin by general considerations allowing us to combine the cocycle
condition, the distributivity of $C^*$-ideals, and the 
Mayer-Vietoris six-term exact sequence into a certain general 
computational tool. Then we use it
to establish the $K$-groups of the aforementioned quantum~$\mathbb{C}P^2$.

To focus attention and for the sake of simplicity, we start by considering
the category of vector spaces. 
Let $J$ be a finite set, and let
\begin{equation}\label{family}
\{\pi^i_j:B_i\longrightarrow B_{ij}=B_{ji}\}_{i,j\in J,\,i\neq j}
\end{equation} 
be a family of  homomorphisms.
In this category, the multi-pullback
of a family~\eqref{family}  can be defined as follows.
\begin{defi}[\cite{GK-P,cm}]
The \emph{multi-pullback} $B^\pi$ of a family~\eqref{family} of  
homomorphisms is defined as
\[
B^\pi:=\left\{\left.(b_i)_i\in\prod_{i\in J}B_i\;\right|\;\pi^i_j(b_i)=\pi^j_i(b_j),\;\forall\, i,j\in J,\, i\neq j
\right\}.
\]
\end{defi}

If $J=\{1,2,3\}$,  a family~\eqref{family} is depicted by the
diagram
\begin{equation}\label{multidiagram}
\xymatrix{
B_1\ar[dr]^{\pi_2^1}\ar@/_2pc/[ddrr]_{\pi_3^1}&&B_2\ar[dl]_{\pi_1^2}\ar[dr]^{\pi_3^2}&&B_3\ar[dl]_{\pi_2^3}\ar@/^2pc/[ddll]^{\pi_1^3}\\
&B_{12}&&B_{23}&\\
&&B_{13}&&},
\end{equation}
and its multi-pullback $B^\pi$ can be interpreted as the limit of this
diagram. (Recall that the limit (colimit) of a diagram is a certain
universal object  together
with morphisms  from it to (to it from) all objects in the diagram.) Furthermore,
one can easily transform the triple-pullback
$B^\pi$ into an iterated pullback:
\begin{lm}\label{vect}
Let $B^\pi$ be the multi-pullback of a family~\eqref{family} for
$J=\{1,2,3\}$. Then the canonical identification of vector 
spaces $V^3\to V^2\times V$ yields an isomorphism from $B^\pi$
to the  pullback vector space $P$ of the top
sub-diagram of the diagram
\begin{equation}\label{iter1}
\xymatrix@-7pt{
&&&P\ar[dll]\ar[drr]&&\\
&P_1\ar[dr]\ar[drrr]^{\gamma}\ar[dl]&&&&B_3\ar[dl]_{\delta}\\
B_1\ar[dr]^{\pi^1_2}&&B_2\ar[dl]_{\pi^2_1}&&P_2\ar[dl]\ar[dr]&\\
&B_{12}&&B_{13}\ar[dr]^{\eta^1}&&B_{23}.\ar[dl]_{\eta^2}\\
&&&&\colim\eqref{multidiagram}&}
\end{equation}
Here all three square sub-diagrams are pullback diagrams, 
$\gamma(b_1,b_2):=(\pi^1_3(b_1),\pi^2_3(b_2))$, 
$\delta(b_3):=(\pi^3_1(b_3),\pi^3_2(b_3))$, and $\eta^1$, $\eta^2$ come 
from the colimit
 of the 
diagram~\eqref{multidiagram}.
\end{lm}
\begin{proof}
By construction, any element of $P$ is a pair
$((b_1,b_2),b_3)\in (B_1\times B_2)\times B_3$ such that
$\pi^1_2(b_1)=\pi^2_1(b_2)$ and
$
(\pi^1_3(b_1),\pi^2_3(b_2))=:\gamma((b_1,b_2))
=\delta(b_3):=(\pi^3_1(b_3),\pi^3_2(b_3)).
$
Hence the re-bracketing map from $B^\pi$ to $P$
is an isomorphism, as claimed. 
\end{proof}

We can still remain in the category of vector spaces to define the
second key concept of this note, notably the cocycle condition.
First, we assume that all maps of a family \eqref{family} 
are surjective. Then, for any distinct $i,j,k$, we put
$B^i_{jk}:=B_i/(\ker\pi^i_j+\ker\pi^i_k)$ and take 
$[\cdot]^i_{jk}:B_i\rightarrow B^i_{jk}$ to be the canonical surjections.
Next, we introduce the family of isomorphisms
\begin{equation}
\pi^{ij}_k:B^i_{jk}\longrightarrow B_{ij}/\pi^i_j(\ker\pi^i_k),\qquad
[b_i]^i_{jk}\longmapsto\pi^i_j(b_i)+\pi^i_j(\ker\pi^i_k).
\end{equation}
Now we are ready for:
\begin{defi}\label{cocycle}
\cite[in Proposition~9]{cm}
 We say that a  family~\eqref{family} of epimorphisms
satisfies the {\em cocycle condition}
 if and only if, for all distinct $i,j,k\in J$, 
\begin{enumerate}
\item $\pi^i_j(\ker\pi^i_k)=\pi^j_i(\ker\pi^j_k)$,
\item the isomorphisms $\phi^{ij}_k:=(\pi^{ij}_k)^{-1}\circ\pi^{ji}_k:B^j_{ik}\rightarrow B^i_{jk}$  satisfy 
$\phi^{ik}_j=\phi^{ij}_k\circ\phi^{jk}_i$.
\end{enumerate}
\end{defi}

\section{A method for computing the $K$-groups of triple-pullback $C^*$-
algebras}

To avoid redundant assumptions, we split this section into an
algebraic and $C^*$-algebraic part. The latter appears as the special
case of the former.

\subsection{Algebras with  distributive lattices of ideals}

From now on we specialize the category of vector spaces
to the category of unital algebras and algebra homomorphisms.
Much of what we do in this subsection is re-casting 
\cite[Corollary~4.3]{hkmz}. However, since our focus is on triple-pullback
algebras, we provide simple direct arguments to spare the reader the
language of sheaves.
First, we slightly extend \cite[Proposition~9]{cm}:
\begin{lm}
Assume that a family~\eqref{family} of algebra epimorphisms 
satisfies the cocycle condition and the kernels of 
these epimorphisms generate a distributive lattice of ideals.
Denote by $\pi_i$, $i\in J$, the restriction of the $i$-th canonical 
projection to the multi-pullback
$B^\pi$ of the family~\eqref{family}.
Then  $B_i\cong B^\pi/\ker{\pi_i}$ for all $i\in J$
 and $B_{ij}\cong B^\pi/(\ker\pi_i+\ker\pi_j)$ for all distinct
 $i,j\in J$.
\end{lm}
\begin{proof}
The existence of isomorphisms  
$B_i\cong B^\pi/\ker{\pi_i}$, $i\in J$, is simply a re-statement
of \cite[Proposition~9]{cm}. To show the existence
of the second family of isomorphisms, we apply \cite[Theorem~7(2)]{hz}
to conclude that, for any distinct $i,j\in J$ and any $b_i\in B_i$,
$b_j\in B_j$, such that $\pi^i_j(b_i)=\pi^j_i(b_j)$, there exists
an element $b\in B^\pi$ such that $\pi_i(b)=b_i$ and $\pi_j(b)=b_j$.
This allows us to prove that the kernels of algebra epimorphisms 
$\pi_{ij}:=\pi^i_j\circ\pi_i=\pi^j_i\circ\pi_j$ are $\ker\pi_i+\ker\pi_j$.
Indeed, if $b\in\ker\pi_{ij}$, then $\pi^i_j(\pi_i(b))=0$ and there 
exists $b'\in B^\pi$ such that $\pi_i(b')=\pi_i(b)$ and $\pi_j(b')=0$.
Therefore, since $b-b'\in\ker\pi_i$ and $b'\in\ker\pi_j$, we
infer that $b\in\ker\pi_i+\ker\pi_j$, as needed. The inclusion
$\ker\pi_i+\ker\pi_j\subseteq\ker\pi_{ij}$ is obvious.
\end{proof}

Combining the above lemma with the \cite[Proposition~8]{hz}, we obtain:
\begin{lm}\label{canlem}
Assume that a family~\eqref{family} of algebra epimorphisms
is such that the restrictions of the  canonical 
projections to the multi-pullback
$B^\pi$ of the family~\eqref{family} are surjective and their kernels
 generate a distributive lattice of ideals.
Then the algebra $B^\pi$ is isomorphic to the
multi-pullback algebra of the  family of canonical surjections
$B^\pi/\ker\pi_i\to B^\pi/(\ker\pi_i+\ker\pi_j)$, $i,j\in J$, $i\neq j$.
\end{lm}

Now we specialize multi-pullbacks to triple-pullbacks, and consider
a special case of the iterated pullback diagram of Lemma~\ref{vect}:
\begin{equation}\label{iterd}
\xymatrix@-10pt@C=-10pt{
&&&{\widetilde P}\ar[dll]\ar[drr]&&\\
&\widetilde P_1\ar[dr]\ar[drrr]^{\widetilde\gamma}\ar[dl]&&&&B^\pi/I_3\ar[dl]_{\widetilde\delta}\\
B^\pi/I_1\ar[dr]&&B^\pi/I_2\ar[dl]&&\widetilde P_2\ar[dl]\ar[dr]&\\
&B^\pi/(I_1+I_2)&&B^\pi/(I_1+I_3)\ar[dr]&&B^\pi/(I_2+I_3).\ar[dl]\\
&&&&B^\pi/(I_1+I_2+I_3)&}
\end{equation}
Here $I_i:=\ker\pi_i$, $i\in\{1,2,3\}$, $\widetilde\gamma(a,b)
:=(a+I_3,b+I_3)$, $\widetilde\delta(c):=(c+I_1,c+I_2)$, and
all three square sub-diagrams are pullback diagrams. To further
abbreviate the notation, we will use $B^\pi_{i}:=B^\pi/I_i$ and
$B^\pi_{ij}:=B^\pi/(I_i+I_j)$ for all distinct $i,j\in\{1,2,3\}$, 
and $B^\pi_{123}:=B^\pi/(I_1+I_2+I_3)$.
\begin{pro}\label{iter}
Assume that a family~\eqref{family} of algebra epimorphisms 
satisfies the cocycle condition and the kernels of 
these epimorphisms generate a distributive lattice of ideals.
Assume also that the kernels of the restrictions of the  canonical 
projections to the multi-pullback
$B^\pi$ of the family~\eqref{family}
generate a distributive lattice of ideals. Take $J=\{1,2,3\}$.
Then the  pullback algebra $\widetilde P$ of diagram~\eqref{iterd} is
isomorphic to $B^\pi$, and
all homomorphisms in this diagram are surjective.
\end{pro}
\begin{proof}
First we take advantage of Lemma~\ref{canlem} to transform
the family~\eqref{family} into its canonical form.  Then we apply
Lemma~\ref{vect} to conclude that the pullback algebra $\widetilde P$ of
the iterated pullback diagram~\eqref{iterd} is isomorphic to the
triple-pullback algebra~$B^\pi$ by the re-bracketing isomorphism. Thus
we can replace $\widetilde P$ by $B^\pi$ in the diagram~\eqref{iterd}.

Since all square sub-diagrams are pullback diagrams and canonical 
quotient maps are surjective, to prove the surjectivity of all 
homomorphisms in the diagram~\eqref{iterd} it suffices to show
the surjectivity of $\widetilde\gamma$ and $\widetilde\delta$. The latter map is surjective
 by \cite[Lemma~2.1]{hkmz}.  It requires a little bit more work to
prove the surjectivity of $\widetilde\gamma$, but our argument is again based on
\cite[Lemma~2.1]{hkmz}. 

Let $(b,c)\in \widetilde P_2$.   Take $a\in B^\pi_{12}$ that is mapped to the same
element in  $B^\pi_{123}$ as $b$ and $c$. It follows from
\cite[Lemma 2.1]{hkmz} that there exists an element $\alpha\in B^\pi_1$ such that $\alpha+I_2=a$ and $\alpha+I_3=b$. Much in the same way,
 we show that there exists an element $\beta\in B^\pi_2$ satisfying 
$\beta+I_1=a$ and $\beta+I_3=c$. By construction, $(\alpha,\beta)\in \widetilde P_1$
 and $\widetilde\gamma((\alpha,\beta))=(b,c)$.
\end{proof}

Finally, since for all distinct $i,j\in\{1,2,3\}$ the identifications 
$B_i\cong B^\pi_i$ and
$B_{ij}\cong B^\pi_{ij}$ are such that together with $\pi^i_j$'s
and canonical quotient maps they form commutative square diagrams,
we immediately conclude:
\begin{cor}\label{cor1}
Assume that a family~\eqref{family} of algebra epimorphisms 
satisfies the cocycle condition and the kernels of 
these epimorphisms generate a distributive lattice of ideals.
Assume also that the kernels of the restrictions of the  canonical 
projections to the multi-pullback
$B^\pi$ of the family~\eqref{family}
generate a distributive lattice of ideals. Take $J=\{1,2,3\}$. Then in the diagram~\eqref{iter1} we can take $\eta^1$ and $\eta^2$ to
be defined as
 $B_{13}\stackrel{\eta^1}{\to}B^\pi_{123}\stackrel{\eta^2}{\leftarrow} B_{23}$, $\eta^i(b):=\widetilde b+ I_1+I_2+I_3$, where $\widetilde b$ is such that $\pi_3^i(\pi_i(\widetilde b))=b$, $i\in\{1,2\}$, and 
 all homomorphisms in this diagram are surjective.
\end{cor}

\subsection{The case of $C^*$-algebras}

Let us assume from now on that all our algebras are unital C*-algebras, 
and morphisms are $C^*$-homomorphisms. Due to the property of C*-ideals
that $I\cap J=IJ$, their kernels always generate
a distributive lattice of ideals, so that  we are in
the special case of the preceding section. On the other hand,
recall that for the 
pullback $C^*$-algebra $A$ of any pair of $C^*$-homomorpisms 
$A_1\stackrel{\alpha^1}{\to} A_{12}\stackrel{\alpha^2}{\leftarrow} A_2$
of which at least one is surjective,
there is the Mayer-Vietoris six-term exact sequence
(e.g., see \cite[Theorem~21.2.2]{bb}
\cite[Section~1.3]{bhms05}, \cite{sch}):
\begin{equation}
\xymatrix@C=50pt{
K_0(A)\ar[r]&K_0(A_1) \oplus K_0(A_2)\ar[r]^{\hskip1cm\alpha^1_*-\alpha^2_*}&K_0(A_{12})\ar[d]\\
K_1(A_{12})\ar[u]&K_1(A_1) \oplus K_1(A_2)\ar[l]_{\alpha^1_*-\alpha^2_*\hskip0.8cm}&K_1(A).\ar[l]}
\end{equation}

Now we can combine Lemma~\ref{vect} with Corollary~\ref{cor1} and
apply three times the above Mayer-Vietoris six-term exact sequence
to infer:
\begin{cor}\label{cor}
Assume that a family~\eqref{family} is a family of $C^*$-epimorphism
and $J=\{1,2,3\}$. Then, if this family satisfies  
the cocycle condition,  there are 
three six-term exact sequences:

\begin{equation}\nonumber
\xymatrix@C=51pt{
K_0(P_1)\ar[r]&K_0(B_1) \oplus K_0(B_2)\ar[r]^{\hskip1cm\pi^1_{2*}-\pi^2_{1*}}&K_0(B_{12})\ar[d]\\
K_1(B_{12})\ar[u]&K_1(B_1) \oplus K_1(B_2)\ar[l]_{\pi^1_{2*}-\pi^2_{1*}\hskip0.8cm}&K_1(P_1),\ar[l]}
\end{equation}
\smallskip

\begin{equation}\nonumber
\xymatrix@C=50pt{
K_0(P_2)\ar[r]&K_0(B_{13}) \oplus K_0(B_{23})\ar[r]^{\hskip1cm\eta^1_*-\eta^2_*}&K_0(B^\pi_{123})\ar[d]\\
K_1(B^\pi_{123})\ar[u]&K_1(B_{13}) \oplus K_1(B_{23})\ar[l]_{\eta^1_*-\eta^2_*\hskip0.8cm}&K_1(P_2),\ar[l]}
\end{equation}
\smallskip

\begin{equation}\nonumber
\xymatrix@C=51pt{
K_0(B^\pi)\ar[r]&K_0(P_1) \oplus K_0(B_3)\ar[r]^{\hskip1cm\gamma_*-\delta_*}&K_0(P_2)\ar[d]\\
K_1(P_2)\ar[u]&K_1(P_1) \oplus K_1(B_3)\ar[l]_{\gamma_*-\delta_*\hskip0.8cm}&K_1(B^\pi).\ar[l]}
\end{equation}
\end{cor}

\newpage
\section{The triple--Toeplitz deformation of $\B CP^2$}
\subsection{$C^*$-algebra}

We consider the case $n=2$ of the multi-Toeplitz deformations
\cite[Section 2]{pmh}
of the complex projective spaces. The $C^*$-algebra of
our quantum projective plane is given as the  triple-pullback 
of the following diagram:
\begin{equation}\label{thefamily}
\xymatrix{
\mathcal{T}\otimes \mathcal{T}\ar[dr]^{\sigma_1}\ar@/_3pc/[ddrr]_{\sigma_2}&&\mathcal{T}\otimes \mathcal{T}\ar[dl]_{\Psi_{01}\circ\sigma_1}\ar[dr]^{\sigma_2}&&\mathcal{T}\otimes \mathcal{T}\ar[dl]_{\Psi_{12}\circ\sigma_2}\ar@/^3pc/[ddll]^{\Psi_{02}\circ\sigma_1}\\
& \mathcal{C}(S^1)\otimes \mathcal{T}&&\mathcal{T}\otimes \mathcal{C}(S^1)&\\
&&\mathcal{T}\otimes \mathcal{C}(S^1)&&}.
\end{equation}
Here $\mathcal{T}$ is the Toeplitz algebra, 
$\sigma\colon\mathcal{T}\to\mathcal{C}(S^1)$ is the symbol
map, $\sigma_1:=\sigma\otimes \id$, $\sigma_2:=\id\otimes\sigma$, and
\begin{align}
\mathcal{C}(S^1)\otimes \mathcal{T}\ni u\otimes z&\stackrel{\Psi_{01}}{\longrightarrow} S(z^{(1)}u)\otimes z^{(0)} \in \mathcal{C}(S^1)\otimes \mathcal{T},\\
\mathcal{C}(S^1)\otimes \mathcal{T}\ni u\otimes z&\stackrel{\Psi_{02}}{\longrightarrow} z^{(0)}\otimes S(z^{(1)}u) \in \mathcal{T}\otimes \mathcal{C}(S^1),\\
\mathcal{T}\otimes \mathcal{C}(S^1) \ni z\otimes u&\stackrel{\Psi_{12}}{\longrightarrow}z^{(0)}\otimes S(z^{(1)}u)\in   \mathcal{T}\otimes \mathcal{C}(S^1),
\end{align}
where $\mathcal{T}\ni z\mapsto z^{(0)}\otimes z^{(1)}\in 
\mathcal{T}\otimes \mathcal{C}(S^1)$ is the coaction dual to the
gauge action on $\mathcal{T}$, and $S(f)(g):=f(g^{-1})$.

\subsection{$K$-theory} The main result of this note is the following:
\begin{thm} The $K$-groups of the triple-Toeplitz deformation of 
$\B CP^2$ are:
$$K_0(\mathcal C(\proj C 2))=\mathbb{Z}^3,\qquad K_1(\mathcal C(\proj C 2))=0.$$
\end{thm}
\begin{proof}
Since the family~\eqref{thefamily} satisfies the cocycle condition
by \cite[Lemma~3.2]{pmh}, we can apply Corollary~\ref{cor} to 
compute the $K$-groups of its triple-pullback $C^*$-algebra
$\mathcal{C}(\proj C 2)$. First, we present  
$\mathcal{C}(\proj C 2)$ as the pullback $C^*$-algebra
of the diagram
\begin{equation}
\xymatrix@C-30pt@R-10pt{
&&&\mathcal{C}(\proj C 2)\ar[dll]\ar[drr]&&\\
&P_1\ar[dr]\ar[drrr]\ar[dl]&&&&\mathcal{T}\otimes\mathcal{T}\ar[dl]\\
\mathcal{T}\otimes\mathcal{T}\ar[dr]_{\sigma_1}&&\mathcal{T}\otimes\mathcal{T}\ar[dl]^{{\Psi_{01}}\circ\sigma_1}&&P_2\ar[dl]\ar[dr]&\\
& \mathcal{C}(S^1)\otimes\mathcal{T}&&\mathcal{T}\otimes \mathcal{C}(S^1)\ar[dr]_{\sigma_1}&&\mathcal{T}\otimes \mathcal{C}(S^1)\ar[dl]^{\sigma_1}\\
&&&&\mathcal{C}(S^1)\otimes \mathcal{C}(S^1)&}
\end{equation}
with all arrows surjective.
We know that 
\begin{equation}\label{knowngroups}
\xymatrix@-20pt{
K_0(\mathcal{T}^{\otimes 2})=\B Z,&K_0(\mathcal{C}(S^1)=\B Z,&K_0(\mathcal{T}\otimes \mathcal{C}(S^1))=\B Z,\\
K_0(\mathcal{T}^{\otimes 2})=0,&K_0(\mathcal{C}(S^1)=\B Z,&K_0(\mathcal{T}\otimes \mathcal{C}(S^1))=\B Z,}
\end{equation}
and that the generators of $K_0$ are 
$[1\otimes 1]\in K_0(\mathcal{T}^{\otimes 2})$ and $[1\otimes 1]\in K_0(\mathcal{T}\otimes\mathcal{C}(S^1))$.

Now the first diagram of Corollary~\ref{cor} becomes
\begin{equation}
\xymatrix{
K_0( P_1)\ar[r]&K_0(\mathcal{T}^{\otimes 2}) \oplus K_0(\mathcal{T}^{\otimes 2})\ar[r]&K_0(\mathcal{T}\otimes\mathcal{C}(S^1))\ar[d]\\
K_1(\mathcal{T}\otimes\mathcal{C}(S^1))\ar[u]&K_1(\mathcal{T}^{\otimes 2}) \oplus K_1(\mathcal{T}^{\otimes 2})\ar[l]&K_1( P_1).\ar[l]}
\end{equation}
After plugging in \eqref{knowngroups}, we obtain
\begin{equation}
\xymatrix@C+20pt{
K_0( P_1)\ar[r]&\mathbb{Z}\oplus\mathbb{Z}\ar@{.>}[r]^{(m,n)\mapsto m-n}&\mathbb{Z}\ar[d]\\
\mathbb{Z}\ar[u]&0\ar[l]&K_1( P_1).\ar[l]}
\end{equation}
This yields $K_0(P_1)=\mathbb{Z}\oplus\mathbb{Z}$ and $K_1( P_1)=0$ because the dotted arrow is onto. 

Next, the second diagram of Corollary~\ref{cor} becomes
\begin{equation}
\xymatrix{
K_0( P_2)\ar[r]&K_0(\mathcal{T}\otimes\mathcal{C}(S^1)) \oplus K_0(\mathcal{C}(S^1)\otimes\mathcal{T})\ar[r]&K_0(\mathcal{C}(S^1)\otimes\mathcal{C}(S^1))\ar[d]\\
K_1(\mathcal{C}(S^1)\otimes\mathcal{C}(S^1))\ar[u]&K_1(\mathcal{T}\otimes\mathcal{C}(S^1)) \oplus K_1(\mathcal{C}(S^1)\otimes\mathcal{T})\ar[l]&K_1( P_2).\ar[l]}
\end{equation}
This is a special case of an exact sequence studied in 
\cite[Section~4]{bhms05}. On the other hand, using a different method,
it was already determined in \cite[Section~3]{hms} that 
 \mbox{$K_0(P_2)=\mathbb{Z}$}
 (generated by $1\in P_2$) and $K_1(P_2)=\mathbb{Z}$. 

Finally, the
last  diagram of Corollary~\ref{cor} becomes
\begin{equation}
\xymatrix{
K_0(\mathcal{C}(\proj C 2))\ar[r]&K_0( P_1) \oplus K_0( \mathcal{T}^{\otimes 2})\ar[r]&K_0( P_2)\ar[d]\\
K_1( P_2)\ar[u]&K_1( P_1) \oplus K_1( \mathcal{T}^{\otimes 2})\ar[l]&K_1(\mathcal{C}(\proj C 2)).\ar[l]}
\end{equation}
Equivalently, we can write it as
\begin{equation}
\xymatrix{
K_0(\mathcal{C}(\proj C 2))\ar[r]&(\mathbb{Z}\oplus \mathbb{Z})\oplus\mathbb{Z}\ar@{.>}[r]&\mathbb{Z}\ar[d]\\
\mathbb{Z}\ar[u]&0\ar[l]&K_1(\mathcal{C}(\proj C 2)).\ar[l]}
\end{equation}
The dotted map is of the form $(m,n,l)\mapsto km+k'n-l$. In particular,
it is onto, so that $K_1(\proj C 2)=0$. Furthermore, the kernel of
this map is $\mathbb{Z}^2$. Combining this with the fact that 
 the short
exact sequences of free modules split, we infer that $K_0(\proj C 2)=\mathbb{Z}^3$. 
\end{proof}
\noindent\textbf{Acknowledgments:}
I would like to thank my advisor P. M. Hajac for his helpful insight in writing  this note.
This work is part of the project \emph{Geometry and Symmetry of Quantum
Spaces} sponsored by the grants 
PIRSES-GA-2008-230836 and 
1261/7.PR~UE/2009/7.

\end{document}